\documentclass[12pt, reqno]{amsart}
\usepackage{amsmath, amsthm, amscd, amsfonts, latexsym, amssymb, graphicx, color}
\usepackage[bookmarksnumbered, colorlinks, plainpages]{hyperref}

\makeatletter \oddsidemargin.9375in \evensidemargin \oddsidemargin
\newtheorem{thm}{Theorem}[section]
\newtheorem{prop}[thm]{Proposition}
\newtheorem{cor}[thm]{Corollary}
\theoremstyle{definition}
\newtheorem{defn}[thm]{Definition}

\theoremstyle{remark}
\newtheorem{rem}[thm]{Remark}
\numberwithin{equation}{section}

\begin{document}

\setcounter{page}{1}

\title[S-small and S-essential submodules]{S-small and S-essential submodules}
\author[Saeed Rajaee]{Saeed Rajaee $^{*}$ }
\thanks{{\scriptsize
\hskip -0.4 true cm MSC(2010): Primary:13A15, 13C99, 16D10, 16D80.
$*$Corresponding author }}
\begin{abstract}
This paper is concerned with S-co-m modules which are a generalization of co-m modules.
In section 2, we introduce the S-small and S-essential submodules of a unitary $R$-module $M$ over a commutative ring $R$ with $1\neq 0$ such that S is a multiplicatively closed subset of $R$.  We prove that if  $M$
is an S-co-m module satisfying the S-DAC and $N\leq M$, then $N\leq ^{S}_{e}M$ if and only
if there exists $I\ll ^{S}R$ such that $s(0:_{M}I)\leq N\leq (0 :_{M}I)$ for some $s\in S$. Let $M$ be a faithful S-strong co-m $R$-module. We prove that if $N\ll ^{S}M$ then there exists an ideal $I\leq ^{S}_{e}R$ such that $s(0 :_{M}I)\leq N\leq (0 :_{M}I)$. The converse is true if $S=\{1\}$and $M$ is a prime module. In section 3, we introduce the S-quasi-copure submodules $N$ of an $R$-module $M$ and investigate some results related to this class of submodules.

\end{abstract}

\maketitle
\section{Introduction}
Throughout this article, $R$ is a commutative ring with $1\neq 0$ and $M$ is a nonzero unital $R$-module. We denote the set of all submodules of $M$ by $L(M)$, and also $L ^{*}(M)=L(M)\setminus \{0, M\}$. A nonempty subset $S$ of $R$ is called a multiplicatively closed subset (briefly, m.c.s.) of $R$ if $0\notin S$,  $1\in S$, and  $ss'\in S$ for all $s, s'\in S$.
Note that $S_{P} = R-P$ is a m.c.s. of $R$ for every
$P\in Spec(R)$. Recently, Sevim et al.
(2019) introduced the notion of $S$-prime submodule which is a generalization of prime submodule and used them to characterize certain class of rings/modules such as prime submodules, simple modules, torsion
free modules, $S$-Noetherian modules and etc. 
In \cite{AA}, Anderson et al. defined the concept
of $S$-multiplication modules
and $S$-cyclic modules which are $S$-versions of multiplication and cyclic
modules and extended many results on multiplication and cyclic modules to
$S$-multiplication and $S$-cyclic modules. An $R$-module $M$ is said to be an $S$-multiplication module if for each submodule $N$ of $M$ there exist $s\in S$ and an ideal $I$ of $R$ such that $sN\subseteq IM\subseteq N$. 
It is easy to see that an $R$-module $M$ is $S$-multiplication if and only if for
each submodule $N$ of $M$, there exists an $s\in S$ such that $sN\subseteq (N : M)M\subseteq N$. If we take $S=\{1_{R}\}$, this definition coincides with the multiplication module definition.

According to \cite[Example 1]{AA}, if $Ann(M)\cap S\neq \emptyset $, then $M$ is an $S$-multiplication module. This implies that if $0\in S$,
then $M$ is trivially $S$-multiplication module. Clearly, every
multiplication module is an $S$-multiplication module and the converse is true if $S\subseteq U(R)$, where $U(R)$ is the set of units in $R$, see \cite[Example 2]{AA}.
Also, $M$ is called an $S$-cyclic $R$-module if there exist $s\in S$ and $m\in M$ with $sM\subseteq Rm\subseteq M$. Every $S$-cyclic module is an $S$-multiplication module, see, \cite[Proposition
5]{AA}. For a prime ideal $P$ of $R$, $M$ is called $P$-cyclic if $M$ is $(R-P)$-cyclic.\\
According to \cite[Proposition 8]{AA}, $M$ is $\mathfrak{m}$-cyclic for each $\mathfrak{m}\in Max(R)$ if and only if $M$ is a
finitely generated multiplication module. We recall that a m.c.s. $S$ of $R$ is said to satisfy maximal multiple condition if there exists $s\in S$ such that $t$ divides $s$ for each $t\in S$.
  
In \cite{AD}, Anderson and Dumitrescu defined the concept of $S$-Noetherian rings
which is a generalization of Noetherian rings and they extended many properties
of Noetherian rings to $S$-Noetherian rings. A submodule $N$ of
$M$ is said to be an $S$-finite submodule if there exists a finitely generated submodule
$K$ of $M$ such that $sN\subseteq K\subseteq N$. Also, $M$ is said to be an $S$-Noetherian module if
its each submodule is $S$-finite. In particular, $R$ is said to be an $S$-Neotherian ring
if it is an $S$-Noetherian $R$-module, i.e., for every ideal $I$ of $R$ there exists a finitely generated ideal
$J$ of $R$ such that $sI\subseteq J\subseteq I$.
  
In \cite{Ed}, Eda Yiliz et al. introduced and
studied $S$-comultiplication modules which are the dual notion of $S$-multiplication
modules. They characterize certain class of rings/modules such as comultiplication
modules, $S$-second submodules, $S$-prime ideals, $S$-cyclic modules in
terms of $S$-comultiplication modules. Let $M$ be an $R$-module and $S \subseteq R$ be a m.c.s of $R$. $M$ is called $S$-comultiplication ($S$-co-m for short), if for each submodule $N$ of $M$, there exist
$s\in S$ and an ideal $I$ of $R$ such that $s(0 :_{M} I)\subseteq N\subseteq (0 :_{M}I)$. In particular, a
ring $R$ is called $S$-co-m ring if it is an $S$-co-m $R$-module. Every $R$-module $M$ with $Ann(M)\cap S = \emptyset $ is trivially an $S$-co-m module. Every co-m module is also an $S$-co-m module. Also the converse is true provided that $S\subseteq U(R)$, see \cite[Example 3]{Ed}.
\section{S-small and S-essential submodules}
In this section we generalize the concepts of small submodule and essential submodule of an $R$-module $M$ to the S-small submodule and S-essential submodule of $M$ such that $S\subseteq R$ is a m.c.s. We provide some useful theorems concerning this new class of submodules. 
\begin{defn}
Let $S$ be a m.c.s. of $R$ and let $M$ be an $R$-module with $N\leq M$.
\begin{itemize}
\item[(i)] $N$ is called an S-small (S-supfluous) submodule of $M$, denoted by $N\ll ^{S}M$, if for every submodule $L$ of $M$ and $s\in S$, $sM\leq N + L$ implies that there exists an $t\in S$ such that $tM\leq L$.
\item[(ii)] $N$ is called an S-essential (S-large) submodule of $M$, denoted by $N\leq^{S}_{e}M$ if for every submodule $L$ of $M$ if $N\cap  L=0$ implies that there exists an $s\in S$ such that $sL=0$.
\item[(iii)] The S-socle of $M$, denoted by $Soc^{S}(M)$ which is the intersection of all S-essential submodules of $M$.
\item[(iv)] The S-radical of $M$, denoted by $Rad^{S}(M)$ which is the sum of all S-small submodules of $M$.
\end{itemize}
If we take $S =\{1_{R}\}$, this definitions coincide with the small and essential submodule definitions.
\end{defn}
\begin{thm}
Let $M$ be an $R$-module with submodules $K\leq N\leq M$ and $S\subseteq R$ be a m.c.s. 
Then the following assertions hold.
\begin{itemize}
\item[(i)] If $K\leq ^{S}_{e}M$ then $K\leq ^{S}_{e}N$ and $N\leq ^{S}_{e}M$.
\item[(ii)] If $K\leq ^{S}_{e}N$ and $M$ is a faithful prime $R$-module then $K\leq ^{S}_{e}M$.
\item[(iii)] Assume that $H\leq M$. If $H\cap K\leq^{S}_{e}M$ then $H\leq^{S}_{e}M$ and $K\leq^{S}_{e}M$.
\item[(iv)] $N\ll ^{S}M$ if and only if $K\ll ^{S}M$ and $N/K\ll ^{S}M/K$.
\end{itemize}
\end{thm}
\begin{proof}
i) Assume $L\leq N$ and $K\cap L=0$. Since $K\leq ^{S}_{e}M$ there exists an $s\in S$ such that $sL=0$. Clearly, $K\leq ^{S}_{e}N$. Now if $L\leq M$ and $N\cap L=0$, then $K\cap L=K\cap (N\cap L)=0$. Since $K\leq ^{S}_{e}M$ there exists an $s\in S$ such that $sL=0$ and hence $N\leq ^{S}_{e}M$.\\
ii) Suppose that $K\leq ^{S}_{e}N$ and $L\leq M$ such that $K\cap L=0$. Then $K\cap (N\cap L)=0$ since $K\leq ^{S}_{e}N$ there exists an $s\in S$ such that $s(N\cap L)=0$. This implies that $s\in Ann_{R}(N\cap L)=Ann_{R}(M)=0$ and therefore $sL=0$.\\
iii) The proof is straightforward by (i).\\
iv) $(\Rightarrow)$ Suppose that $sM\leq K+L$ for some $L\leq M$ and $s\in S$. This implies that $sM\leq N+L$ since $N\ll ^{S}M$ hence there exists an $t\in S$ such that $tM\leq L$ this conclude that $K\ll ^{S}M$. Now let $s(M/K)\leq N/K+L/K$ for some $s\in S$ and $L/K\leq M/K$. Then $s(M/K)=(sM+K)/K\leq (N+L)/K$ and hence $sM\leq sM+K\leq N+L$. Since $N\ll ^{S}M$ there exists an $t\in S$ such that $tM\leq L$. It conclude that $tM+K\leq L+K=L$ and hence $t(M/K)=(tM+K)/K\leq L/K$. This implies that $N/K\ll ^{S}M/K$.\\
$(\Leftarrow )$ Suppose that $sM\leq N+L$ for some $K\leq L\leq M$ and $s\in S$. Then $sM+K\leq (N+L)+K=N+L$ and hence
\[s(M/K)=(sM+K)/K\leq (N+L)/K=N/K+L/K.\]
Since $N/K\ll ^{S}M/K$  therefore there exists an $t\in S$ such that $t(M/K)\leq L/K$. Then $tM\leq tM+K\leq L$ and hence $N\ll ^{S}M$.
\end{proof}
\begin{defn}
Let $S\subseteq R$ be a m.c.s. and let $M$ be an $R$-module.
\begin{itemize}
\item[(i)]  $M$ satisfies the S-double annihilator condition (S-DAC for short) if for
each ideal $I$ of $R$ there exists an $s\in S$ such that $sAnn_{R}((0 :_{M}I))\subseteq I$, see \cite[Definition 2.14]{Far}.
\item[(ii)] $M$ is an S-strong co-m module if $M$ is an S-co-m
$R$-module which satisfies the S-DAC, , see \cite[Definition 2.15]{Far}.
\item[(iii)] A submodule $N$ of $M$ is called an S-direct summand of $M$ if there exist a
submodule $K$ of $M$ and $s\in S$ such that $sM = N + K$, \cite[Definition 2.8]{Far}.
\item[(iv)] $M$ is said to be an S-semisimple module if every submodule of $M$ is an S-direct
summand of $M$, see \cite[Definition 2.9]{Far}.
\end{itemize}
\end{defn}
\begin{prop}
Let $M$ be a faithful S-strong co-m $R$-module.
\begin{itemize}
\item[(i)] If $N\ll ^{S}M$ then there exist $I\leq ^{S}_{e}R$ and $t\in S$ such that $s(0 :_{M}I)\leq N\leq (0 :_{M}I)$. The converse is true if $S=\{1\}$and $M$ is a prime module.
\item[(ii)] If $M$ is an S-semisimple $R$-module then the assertion (i) satisfies.
\end{itemize}
\end{prop}
\begin{proof}
i) Assume that $N\ll^{S}M$ since $M$ is an S-co-m module there exist an
ideal $I$ of $R$ and an $t\in S$ such that $t(0 :_{M}I)\leq N\leq (0 :_{M}I)$. Suppose that $I\cap J = 0$ for some ideal $J$ of $R$. By virtue of \cite[Lemma 2.16 (b)]{Far} there exists an $s\in S$ such that
\begin{align*}
N +(0 :_{M}J)&\geq t(0 :_{M}I) + (0 :_{M}J)\geq t(0 :_{M}I) + t(0 :_{M}J)\\
&\geq st(0 :_{M}I \cap J)= stM.
 \end{align*}
Take $s'=st\in S$. Since $N\ll^{S}M$ hence $s'M\leq N+(0:_{M}J)$ implies that there exists an $s''\in S$ such that $s''M\leq (0:_{M}J)$. This conclude that $s''J\subseteq Ann_{R}(M)=0$ and then $I\leq ^{S}_{e}R$.\\ 
Conversely, let $N\in L(M)$ such that $t(0 :_{M}I)\leq N\leq (0 :_{M}I)$ for some $t\in S$ and $I\leq ^{S}_{e}R$. Assume that $K\leq M$ such that $sM\leq N+K$ for some $s\in S$. We must show that there exists an $x\in S$ such that $xM\leq K$. Since $M$ is an S-co-m module there exist $t'\in S$ and an ideal $J$ of $R$ such that $t'(0:_{M}J)\leq K\leq (0:_{M}J)$.
Set $t=1$ in equality $t(0 :_{M}I\cap J)\leq (0 :_{M}I) + (0 :_{M}J)$. It conclude that $(0:_{M}I\cap J)=(0:_{M}I)+(0:_{M}J)\geq N+K\geq sM$. Therefore $I\cap J\subseteq Ann_{R}(sM)=Ann_{R}(M)=0$ and since $I\leq ^{S}_{e}R$, there exists an $s'\in S$ such that $s'J=0$. It conclude that $s'M\leq (0:_{M}J)$. Take $x=s't'$, then $xM=s't'M\leq t'(0:_{M}J)\leq K$.\\
ii) Since $M$ is an S-semisimple module hence every submodule of $M$ is an S-direct
summand of $M$. Therefore for every submodule $N$ of $M$ there exist a
submodule $K$ of $M$ and $s\in S$ such that $sM = N + K$. This implies the assertion (i).
\end{proof}
\begin{thm}
Let $M$ be a torsion-free S-strong co-m module and let $N\leq M$. Then $N\leq ^{S}_{e} M$ if and only if there exist $I\ll ^{S}R$ and an $s\in S$ such that $s(0 :_{M} I)\leq N\leq (0 :_{M} I)$.
\end{thm}
\begin{proof}
$(\Rightarrow)$ Suppose that $N\leq ^{S}_{e} M$. Since $M$ is an S-co-m module, there exist an
ideal $I$ of $R$ and an $s\in S$ such that $s(0 :_{M}I)\leq N\leq (0:_{M}I)$. Assume that $tR\leq I+J$ for some ideal $J$ of $R$ and an $t\in S$, then
\begin{align*}
N\cap (0 :_{M}J)&\leq (0 :_{M}I)\cap (0 :_{M}J) = (0 :_{M}I+J)\leq (0 :_{M}tR) = 0.
\end{align*}
Since $N\leq ^{S}_{e}M$ there exists an $t'\in S$ such that $t'(0 :_{M}J) = 0$ and therefore $t'\in Ann_{R}((0:_{M}J))$. Since $M$ satisfies the S-DAC there exists an $t''\in S$ such that
$t't''\in t''Ann_{R}((0 :_{M}J))\subseteq J$. Take $x=t't''\in S$, then $xR\subseteq J$ and the proof is complete.\\
$(\Leftarrow)$ Assume that there exists an ideal $I\ll ^{S}R$ such that $s(0 :_{M} I)\leq N\leq (0 :_{M} I)$ for some $s\in S$. Let $L\leq M$ and $N\cap L=0$. We must show that there exists an $y\in S$ such that $yL=0$. Since $M$ is an S-co-m $R$-module there exist an ideal $J$ of $R$ and an $t\in S$ such that
$t(0 :_{M}J)\leq L\leq (0:_{M}J)$. This implies that 
\begin{align*}
0&=N\cap L\geq s(0:_{M}I)\cap t(0:_{M}J)\geq st((0:_{M}I)\cap (0:_{M}J))\\
&=st(0:_{M}I+J).
\end{align*}
Therefore $st\in Ann_{R}(0:_{M}I+J)$. Since $M$ satisfies S-DAC hence there exists an $t'\in S$ such that 
$t'Ann_{R}(0:_{M}I+J)\subseteq I+J$. Take $x=stt'\in S$. This conclude that $x\in I+J$ and then $xR\leq I+J$. Since  $I\ll ^{S}R$ then there exists an $y\in S$ such that $yR\subseteq J$. This implies that $y\in J$ and hence $yL\leq y(0:_{M}J)=0$.
\end{proof}
\begin{cor}
Let $M$ be a torsion-free S-strong co-m $R$-module and let $N\leq M$. Then $Soc^{S}(M)\leq (0:_{M}Rad^{S}(R))$.
\end{cor}
\begin{proof}
The proof is clear by Theorem 2.5, since
\[Soc^{S}(M)=\bigcap _{N\leq _{e}^{S}M}N\leq \bigcap _{I\ll ^{S}M}(0:_{M}I)=(0:\sum _{I\ll ^{S}M}I)=(0:_{M}Rad^{S}(R)).\]
\end{proof}
\section{S-quasi copure submodules}
In this section we define the concept of S-quasi copure submodule of an $R$-module $M$ and provide some results concerning this new class of submodules.
\begin{defn}
Let $S$ be a m.c.s. of $R$ and $P$ a submodule of $M$ with $(P :_{R} M)\cap S =\emptyset $. Then the submodule $P$ is called an $S$-prime submodule if there exists $s\in S$, and whenever $am\in P$, then $sa\in (P :_{R} M)$ or $sm\in P$ for each $a\in R$, $m\in M$. Particularly, an ideal $I$ of $R$ is called an $S$-prime ideal if $I$ is an $S$-prime submodule of $R$-module $R$. We denote the sets of all prime
submodules and all $S$-prime submodules of $M$ by $Spec(M)$ and $Spec_{S}(M)$, respectively.
\end{defn}
Note that for every $P\in Spec(M)$ such that  $(P:_{R}M)\cap S=\emptyset $, then $P\in Spec_{S}(M)$ since $1\in S$. Also, if we take $S\subseteq U(R)$, where $U(R)$ denotes the set of units in $R$, the notions of $S$-prime submodules and prime submodules are equal.
\begin{rem}
 For any submodule $N$ of an $R$-module $M$, we define $V^{S}(N)$ to be the set of all $S$-prime submodules
of $M$ containing $N$. Also the $S$-radical of a submodule $N$ of $M$ is the intersection of all $S$-prime submodules of $M$ containing $N$, denoted by $rad^{S}(N)$ therefore $rad^{S}(N)=\cap V^{S}(N)$. If $N$ is not contained in any $S$-prime submodule of $M$, then $rad^{S}(N) = M$. 
\end{rem}
\begin{defn}
Let $S$ be a m.c.s. of $R$ and let $M$ be an $R$-module.
\begin{itemize}
\item[(i)] A submodule $N$ of $M$ is said to be $S$-pure if there exists an $s\in S$ such that
$s(N \cap IM)\subseteq IN$ for every ideal $I$ of $R$. Also, $M$ is said to be fully $S$-pure if every submodule of $M$ is $S$-pure.
\item[(ii)] A submodule $L$ of $M$ is $S$-copure if there exists an $s\in S$ such
that $s(L :_{M}I)\subseteq L + (0 :_{M} I)$ for every ideal $I$ of $R$. We will denote the set of all
$S$-copure submodules of $M$ by $C^{S}(M)$. An
$R$-module $M$ is fully $S$-copure if every submodule of $M$ is $S$-copure, i.e., $L(M)=C^{S}(M)$. For a
submodule $N$ of an $R$-module $M$, we will denote the set of all
$S$-copure $S$-prime submodules of $M$ containing $N$ by $CV^{S}(N)$. Equivalently, 
$CV^{S}(N)=V^{S}(N)\cap C^{S}(M)$. If $N$ is not contained
in any $S$-prime $S$-copure submodule of $M$, then we put $CV^{S}(N) = M$.
\item[(iii)] A submodule $N$ of $M$ is called a weak $S$-copure submodule if every prime submodule $P$ of $M$ containing $N$ is an $S$-copure submodule of $M$, i.e., $V(N)\subseteq C^{S}(M)$. We will denote the set of all this submodules of $M$ by $C^{S}_{w}(M)$.
\item[(iv)] A submodule $N$ of $M$ is called an $S$-quasi-copure submodule if every $S$-prime submodule $P$ of $M$ containing $N$ is an $S$-copure submodule of $M$. Equivalently, if $V^{S}(N)\subseteq C^{S}(M)$ hence $V^{S}(N) = CV^{S}(N)$.
We will denote the set of all this submodules of $M$ by $C^{S}_{q}(M)$.
\end{itemize}
\end{defn}
\begin{thm}
Let $S\subseteq R$ be a m.c.s. and let $M$ be an S-co-m module. Then the following assertions hold.
\begin{itemize}
\item[(i)] If $N\in C^{S}(M)$ then $M/N$ is an S-co-m $R$-module.
\item[(ii)] If $N\in C^{S}(M)$ then for every $s\in S$, $M/sN$ is an S-co-m $R$-module.
\end{itemize}
\end{thm}
\begin{proof}
i) Let $K/N\leq M/N$. Since $M$ is an S-co-m $R$-module, there exist an
ideal $I$ of $R$ and an $s\in S$ such that $s(0 :_{M}I)\subseteq K\subseteq (0 :_{M}I)$. Then 
\[s\frac{(0:_{M}I)}{N}=\frac{s(0:_{M}I)+N}{N}\subseteq \frac{K+N}{N}=\frac{K}{N}\subseteq \frac{(0:_{M}I)}{N}.\]
Hence, $M/N$ is an S-co-m module. 
ii) This follows by part (i) and \cite[Proposition 2.7 (c)]{Far}.
\end{proof}
\begin{thm}
Let $M$ be an $R$-module. If $S\subseteq T$ are m.c.s. of $R$ and $N, K\in L(M)$ such that $N\subseteq K$. The following statements hold.
\begin{itemize}
\item[(i)] If $N\in C_{w}^{S}(M)$ then $K\in C_{w}^{S}(M)$.
\item[(ii)] If $N\in C_{w}^{S}(M)$ then $K/N\in C_{w}^{S}(M/N)$. 
\item[(iii)] Assume that $M$ is a distributive module. If $N, K\in C^{S}(M)$, then $N\cap K\in C^{S}(M)$. Moreover, if $V(N)$ is a finite set and $N\in C^{S}(M)$, then $rad(N)\in C^{S}(M)$.
\item[(iv)] Suppose that $M$ is a multiplication module, and $N, K\in L(M)$. If
$P\in V(NK)$ such that $(P:M)\cap S=\emptyset $, then there exists a $s\in S$ such that $sN\subseteq P$ or $sK\subseteq P$.
\item[(v)] $C^{S}_{q}(M)\subseteq C^{T}_{q}(M)$.
\item[(vi)] If $N\in C^{S}_{q}(M)$, then for every $\mathfrak{p}\in Spec(R)$, $N_{\mathfrak{p}}\in C^{S_{\mathfrak{p}}}_{q}(M_{\mathfrak{p}})$. 
\end{itemize}
\end{thm}
\begin{proof}
i) It is clear.\\
ii) Suppose that $P\in Spec_{S}({}_{R}M)$ and $N\leq K\leq P$ then by \cite[Corollary 2.8 (ii)]{SAT},
 $P/N\in Spec _{S}({}_R(M/N))$. Then by \cite[Theorem 2.6 (c)]{Far} since $P$ is an S-copure submodule of $M$ hence $P/N$ is an
 S-copure submodule of $M/N$ such that $K/N\leq P/N$.\\
iii) Since $N, K\in C^{S}(M)$ hence there exist $s_{1}, s_{2}\in S$ such that for every ideal $I$ of $R$, 
$s_{1}(N:_{M}I)\subseteq N+(0:_{M}I)$ and also $s_{2}(K:_{M}I)\subseteq K+(0:_{M}I)$. Take $s=s_{1}s_{2}\in S$, then for every $a\in R$,
\begin{align*}
s(N\cap K:_{M}a)&=s_{1}s_{2}((N:_{M}a)\cap (K:_{M}a))\\
& \subseteq s_{1}(N:_{M}a)\cap s_{2} (K:_{M}a)\\
&\subseteq (N+(0:_{M}a)\cap (K+(0:_{M}a))\\
&=(N\cap K)+(0:_{M}a).
\end{align*}
Therefore by \cite[Theorem 2.12]{Far}, we include that $N\cap K\in C^{S}(M)$.The second part is clear by induction on $|V(N)|< \infty $.\\
iv) Suppose that $P\in V(NK)$ and $(P :_{R}M)\cap S =p\cap S= \emptyset $ where $p=(P:M)\in Spec(R)$. By \cite[Proposition 2.2]{SAT}, $P\in V^{S}(NK)$. Assume that $N=IM$ and $K=JM$ for some ideals $I$ and $J$ of $R$. By virtue of \cite[Lemma 2.5]{SAT}, since $P$ is an $S$-prime submodule of $M$ and $P\supseteq NK=IJM$ hence there exists an $s\in S$ such that $sIJ\subseteq (P:_{R}M)$ or $sM\subseteq P$. If $sM\subseteq P$, then $s\in (P:_{R}M)$ which is impossible. This implies that $sIJ\subseteq (P:_{R}M)$ for some $s\in S$. By \cite[Proposition 2.9]{SAT}, since $M$ is a multiplication module therefore $P\in Spec_{S}(M)$ if and only if  $p=(P:_{R}M)\in Spec_{S}(R)$. Since $sIJ\subseteq p$, then by \cite[Corollary 2.6]{SAT}, there exists an $t\in S$ such that
$t(sI)=tsI\subseteq p$ or $tsJ\subseteq tJ\subseteq p$. Therefore either $ts(IM)=tsN\subseteq pM=P$ or $tsJM=tsK\subseteq pM=P$. Take $s'=ts$ then the proof is complete.\\
v) Since $M$ is an $S$-multiplication module, then by \cite[Proposition 1]{AA}, $M$ is also a
$T$-multiplication module. Assume that $P$ is an $S$-prime submodule of $M$ contaning $N$, then by \cite[Proposition 2.2 (ii)]{SAT} $P$ is an $T$-prime submodule of $M$ containing $N$ in the case $(P :_{R}M)\cap T=\emptyset $. If $N\in C^{S}_{q}(M)$, then $P$ is an $S$-copure submodule of $M$ and by \cite[Proposition 2.7 (a)]{Far}, $P$ is an $T$-copure submodule of $M$ containing $N$. This implies that $N\in C^{T}_{q}(M)$.\\
vi) Suppose that $Q_{\mathfrak{p}}\in V^{S_{\mathfrak{p}}}(N_{\mathfrak{p}})$ is an $S_{\mathfrak{p}}$-prime submodule of $M_{\mathfrak{p}}$ as an $R_{\mathfrak{p}}$-module containing $N_{\mathfrak{p}}$.  Since $N\in C^{S}_{q}(M)$ hence every $S$-prime submodule $Q$ of $M$ containing $N$ is $S$-copure, then by \cite[Proposition 2.13]{Far}, $Q_{\mathfrak{p}}$ is a $S_{\mathfrak{p}}$-copure submodule of $M_{\mathfrak{p}}$. 
\end{proof}
We recall that the saturation $S^{*}$ of $S$ is defined as $S^{*}=\{x\in R : \frac{x}{1}\in U(S^{-1}R)\}$. Obviously, $S^{*}$ is a m.c.s. of $R$ containing $S$, see \cite{Gil}.
\begin{thm}
Let $S$ be a m.c.s. of $R$. The following assertions hold.
\begin{itemize}
\item[(i)]  $C^{S}_{q}(M)\subseteq C^{S^{*}}_{q}(M)$.
\item[(ii)] Assume that $M$ is a finitely generated faithful multiplication module then $N=IM\in C^{S}_{q}(M)$ if and only if $I\in C^{S}_{q}(R)$ such that $N=IM$ for some ideal $I$ of $R$. Furthermore, for every $P\in Spec_{S}(M)$
such that $(P : M)\cap S =\emptyset $, then $rad^{S}(M)=rad^{S}(R)M$.
\end{itemize}
\end{thm}
\begin{proof}
i) It is clear.\\
ii) Assume that $\mathfrak{p}\in Spec_{S}(R)$ such that $\mathfrak{p}\supseteq I$. We must show that $\mathfrak{p}$ is an S-copure ideal of $R$. Since $M$ is a multiplication module by \cite[Proposition 2.9 (ii)]{SAT}, $P=\mathfrak{p}M\in Spec_{S}(M)$. By hypothesis since $N=IM\in C^{S}_{q}(M)$ and $P=\mathfrak{p}M\geq N=IM$ this conclude that $P$ is an S-copure submodule of $M$. Therefore there exists an $s\in S$ such that 
$s(P:_{M}\mathfrak{a})\leq P+(0:_{M}\mathfrak{a})$ for each ideal $\mathfrak{a}$ of $R$. We prove that 
$s(\mathfrak{p}:_{R}\mathfrak{a})\subseteq \mathfrak{p}+(0:_{R}\mathfrak{a})$ for each ideal $\mathfrak{a}$ of $R$. We conclude that 
\begin{align*}
s(P:_{M}\mathfrak{a})&=s(\mathfrak{p}M:_{M}\mathfrak{a})=s(\mathfrak{p}:_{R}\mathfrak{a})M\leq P+(0:_{M}\mathfrak{a})\\
&=\mathfrak{p}M+((0:_{M}\mathfrak{a}):_{R}M)M\\
&=(\mathfrak{p}+(0:_{R}\mathfrak{a}))M.
\end{align*}
Since $M$ is a cancellation module therefore $s(\mathfrak{p}:_{R}\mathfrak{a})\subseteq \mathfrak{p}+(0:_{R}\mathfrak{a})$. The converse is similar. By \cite[Theorem 2.11]{SAT}, we have
\[rad^{S}(M)=\bigcap _{Ann(M)\subseteq I\in Spec^{S}(R)}IM=rad^{S}(R)M.\]
\end{proof}
\section*{Acknowledgements}
The author is grateful to the referee for helpful suggestions which
have resulted in an improvement to the article. This article was
supported by the Grant of Payame Noor University of Iran.

\end{document}